\documentclass[11pt, oneside]{amsart}
\RequirePackage{amsmath}
\usepackage{amsthm}

\usepackage[utf8]{inputenc}
\usepackage{hyperref}
\usepackage{url}

\usepackage{nomencl}
\makenomenclature

\usepackage{mathtools}
\usepackage{enumerate}
\usepackage{cite}
\usepackage[shortlabels]{enumitem}
\usepackage{tikz-cd}
\usepackage{stmaryrd}

\usepackage{latexsym}
\usepackage{amsmath,amssymb,amsthm,amscd}
\usepackage{color}
\usepackage[nolabel, final]{showlabels}
\title{Partial Bloch--Kato Selmer groups of $B$-pairs as delta functors}
\author{Rustam Steingart}
\address{Ruprecht-Karls-Universität Heidelberg,
	Mathematisches Institut,Im Neuenheimer Feld 205, D-69120 Heidelberg}
\email{rsteingart@mathi.uni-heidelberg.de}
\date{\today}

\theoremstyle{plain}
\newtheorem{thm}{Theorem}[subsection]
\newtheorem{lem}[thm]{Lemma}
\newtheorem{rem}[thm]{Remark}
\newtheorem{prop}[thm]{Proposition}
\newtheorem{cor}[thm]{Corollary}

\newtheorem*{cor*}{Corollary}
\newtheorem{introtheorem}{Theorem}

\theoremstyle{definition}
\newtheorem{defn}[thm]{Definition}

\newcommand{\NN}{\mathbb{N}}

\newcommand{\Gal}{\operatorname{Gal}}
\newcommand{\Hom}{\operatorname{Hom}}

\newcommand{\ZZ}{\mathbb{Z}}

\newcommand{\QQ}{\mathbb{Q}}

\newcommand{\CC}{\mathbb{C}}
\newcommand{\C}{\mathbb{C}}

\newcommand{\fA}{\mathbf{A}}

\newcommand{\fB}{\mathbf{B}}

\newcommand{\cR}{\mathcal{R}}

\newcommand{\id}{\operatorname{id}}

\newcommand{\Bdr}{\fB_{\mathrm{dR}}}
\newcommand{\Bcris}{\fB_{\mathrm{cris}}}

\DeclarePairedDelimiter\abs{\lvert}{\rvert}%
\begin{document}
	\maketitle

	\begin{abstract}
		In this article we revisit the partial Selmer groups introduced by Ding in cohomological degree one. On the subcategory of partially de Rham positive $B$-pairs we extend them to higher cohomological degree and show that the resulting groups form a cohomological delta functor satisfying a variant of the Euler--Poincaré characteristic formula and Tate duality. 
	\end{abstract}
	\section*{Introduction} 
	Let $K/\QQ_p$ be finite and $K\subset E$ a finite extension containing a normal closure. 
	The fundamental exact sequence of $p$-adic Hodge--Theory 
	$$0 \to \QQ_p  \to \fB_e \oplus \Bdr^+ \to \Bdr \to 0$$
	 allows us to assign to a $p$-adic representation $V$ of $G_K$ $B$-pair 
	 $W:=(V \otimes_{\QQ_p}\fB_e,V\otimes_{\QQ_p}\Bdr^+)$ and if $V$ is an $E$-linear representation then $W$ is naturally an $E$-$B$-pair, i.e., a compatible tuple $(W_e,W_{\mathrm{dR}}^+)$ with $W^{??}_?$ a module over $E \otimes_{\QQ_p}\fB^{??}_{?}$ equipped with an action of $G_K.$
	In modern terms, the category of $B$-pairs is equivalent to the category of $G_K$-equivariant vector bundles on the Fargues--Fontaine curve but for the purposes of this article, we will stick to the $B$-pair viewpoint. 
	 From the decomposition $E\otimes_{\QQ_p}\Bdr= \prod_{\sigma\colon K\to E} E \otimes_{\sigma,K} \Bdr$ one obtains a decomposition of $W_{\mathrm{dR}}$ into components $W_{\mathrm{dR},\sigma}$ indexed by the set of embeddings $\Sigma_K = \Hom_{\QQ_p}(K,E).$ Recall that (global) representations coming from geometry  are de Rham above $p.$ The property of being de Rham is not independent of the place above $p,$ i.e., the embedding $\sigma\colon K\to E$ in the sense that the question whether 
	 $$\Bdr \otimes_E H^0(G_K,\Bdr\otimes_{K,\sigma}V) \to \Bdr\otimes_{K,\sigma} V$$ is an isomorphism depends on $\sigma.$ 
	 To give a specific example we can consider the extension 
	 $$\begin{pmatrix}
	 	1 & \log(\chi_{\mathrm{LT}})\\
	 	0 & 1
	 \end{pmatrix}$$
	 with the Lubin--Tate character $\chi_{\mathrm{LT}}$ of $E$ viewed as a representation of $G_E.$
	Using that the Hodge--Tate weights of $\chi_{LT}$ are $1$ at the identity
	 and $0$ at the other embeddings, one concludes that this extension is de Rham at all embeddings except the identity. \\
	  In \cite{ding2017protect} Yiwen Ding introduced 
	for an $E$-$B$-pair $W=(W_e,W_{dR}^+)$ and a subset $J\subseteq \Hom_{\QQ_p}(K,E)$ the partial Bloch--Kato Selmer group 
	$$H^1_{g,J}(G_K,W) := \ker[H^1(G_K,W) \to \prod_{\sigma \in J} H^1(G_K,W_{dR,\sigma})].$$  This subgroup captures precisely the extensions of the trivial representation by $W$ which are de Rham at the prescribed set of embeddings $J\subset \Sigma_K.$
	The most restrictive condition would be $J=\Sigma_K,$ which recovers the usual Bloch--Kato Selmer group $H^1_g(G_K,W).$\\
	We will consider these groups on the subcategory $\mathcal{BP}_E^{J,\mathbf{n}}$ for a multi-index $\mathbf{n}=(n_\sigma)_{\sigma\in J} \in \mathbb{N}^J,$ consisting of $B$-pairs which are de Rham of Hodge--Tate weight in $[-n_{\sigma},0]$ for all $\sigma\in J.$
	This is a generalisation of the category of ``analytic'' $B$-pairs consisting of those with Hodge--Tate weight $0$ at all but one embedding. These correspond to Lubin--Tate $(\varphi,\Gamma)$-modules over the Robba ring with locally $K$-analytic action introduced by Berger in \cite{Berger2016} (for details and definitions we refer to section \ref{sec:analytic}).
	We show that one can define suitable $H^0_{g,J}(-),H^2_{g,J}(-)$ to extend the definition of Ding beyond $H^1$ in a way which satisfies a series of nice properties. 
	The groups $H^i_{g,J}(-)$ can be defined for any $B$-pair not just for those belonging to the category $\mathcal{BP}_E^{J,\mathbf{n}},$ but their behaviour is less structured. For example, if a representation $V$ is de Rham with all weights $\geq1$ then any extension of the trivial representation by $V$ is de Rham. 
	\begin{introtheorem} 
		\label{thm:mainresultintro} Let $K/\QQ_p$ be finite, let $E$ be a normal closure of $K,$ let $J\subseteq \Sigma_K= \Hom_{\QQ_p}(K,E),$ and let $\mathbf{n} \in \NN^J.$
		Let  $\delta = [x \mapsto \prod_{\sigma \in J} \sigma(x)^{n_\sigma+1}]$ viewed as a rank one object of $\mathcal{BP}_E.$ 
		Let $\chi_D:=\chi_{cyc}/\delta.$ Then:
		\begin{enumerate}[(i)]
			\item The group $H^1_{g,J}(W)$ agrees with $\operatorname{Ext}_{\mathcal{BP}^{J,\mathbf{n}}_E}(B_E,W).$
			\item $(H^n_{g,J}(W))_{n \in \NN}$ is a delta functor. 
			\item $H^2_{g,J}(\chi_D)= H^2(\chi_{cyc})\cong E$ 
			\item We have the following Euler--Poincaré formula:
			$$-\operatorname{rank}(W)([K:\QQ_p]-\abs{J})=\sum_{i=0}^2 (-1)^i \dim_EH^i_{g,J}(W).$$ 
	\item	For every $W \in \mathcal{BP}_E^{J,\mathbf{n}}$ we have $W^*(\chi_D) \in \mathcal{BP}_E^{J,\mathbf{n}}$ and there exists a canonical perfect pairing
		$$H^i_{g,J}(W) \times H^{2-i}_{g,J}(W^*(\chi_D)) \to  E.$$ 
			
			\item For $J=\emptyset$ this specialises to Galois cohomology of $B$-pairs.  
			\item For $K=E$ and $J = \Sigma_K \setminus \sigma_0$ and $\mathbf{n}= (0,\dots,0)$ this specialises to analytic cohomology, i.e., $H^i_{g,J}(W) \cong H^i_{an}(M(W)),$ where $M(W)$ is the analytic $(\varphi_K,\Gamma_K)$-module over $\cR_K$ attached to $W.$ (with respect to the embedding $\sigma_0$).  
		\end{enumerate}
	\end{introtheorem}
	Note that any object of  $\mathcal{BP}_E^{J,\mathbf{n}}$ also belongs to 
$\mathcal{BP}_E^{J,\mathbf{n}'}$ if $n'_\sigma \geq n_\sigma.$
The groups $H^i_{g,J}(-)$ for $i=0,1$ are independent of the choice
of $\mathbf{n}$ and can be realised as subsets of $H^i(G_K,-).$ The group $H^2_{g,J}(W)$ is defined as $H^2(W(\delta))$ and hence depends on $\mathbf{n}$ a priori but turns out to be canonically isomorphic for different choices of $\mathbf{n}.$
However, if we want to dualise $H^i_{g,J}(W)$ we realise that we need to twist $W^*(1)$ by a suitable character to obtain an object of $\mathcal{BP}_E^{J,\mathbf{n}},$ which is precisely why $\chi_D$ is chosen in this way and why we work with a fixed choice of $\mathbf{n}.$
For different choices of $\mathbf{n}$ such that $W$ belongs to $\mathcal{BP}_E^{J,\mathbf{n}}$  it is possible to compare the resulting dualities. \\
	 Our Theorem is an improvement on the state of the art concerning analytic cohomology. 	These have been previously studied extensively in (among others) \cite{FX12},\cite{colmez2016representations},\cite{BFFanalytic}.  
	It is suggested in the introduction of \cite{FX12} that the viewpoint of \cite{nakamura2009classification} is less suitable for applications. 
Our results suggest that the opposite is the case, in fact the $B$-pair viewpoint allows us to give an arithmetic description in terms of Galois cohomology and show the Euler--Poincaré formula and Tate duality in full generality (for field coefficients) which were previously only known in the trianguline case after base change to a transcendental extension (cf. \cite{steingart2024iwasawa, MSVW}). We can also apply the main result for $J=\Sigma_K$ to obtain new formulae for the dimension of the Bloch--Kato Selmer group $H^1_g.$ 
	If $W \in \mathcal{BP}_E^{J,\mathbf{n}}$ then $W$ also belongs to $\mathcal{BP}_E^{J',\mathbf{n}}$ for any $J'\subset J.$ If for example $W$ is positive and de Rham, i.e., belongs to $\mathcal{BP}_E^{\Sigma_K,\mathbf{n}}$ for some $\mathbf{n},$ then, if we write $\Sigma_K = \bigcup_{i=0}^{[K:\QQ_p]} J_i$ as a nested union of subsets such that $\abs{J_i} = i$ we obtain a flag
	$$0 \subseteq H^1_{g}(W) = H^1_{g,J_d}(W) \subsetneq \dots H^1_{g,J_{1}}(W) \subsetneq  H^1_{g,\emptyset}(W) = H^1(W),$$
	where $d=[K:\QQ_p].$
	It would be interesting to define a filtration on $\mathbf{R}\Gamma(G_K,W)$ which induces this flag in degree one to obtain a spectral sequence. 
	\section*{Acknowledgements}
This research was supported by Deutsche Forschungsgemeinschaft (DFG) - Project number 536703837, which allowed me to carry out my research at the UMPA of the ENS de Lyon. I would like to thank the institution and in particular Laurent Berger for his guidance and many fruitful discussions. I also thank Marlon Kocher for his remarks on a previous version of this article. 
	\section{Preliminaries}
	We summarise the main classical viewpoints: $B$-pairs, cyclotomic $(\varphi,\Gamma)$-modules. The discussion of ($K$-analytic) Lubin--Tate $(\varphi,\Gamma)$-modules is postponed to section \ref{sec:analytic}. 
	\subsection{$B$-pairs} 
	Let $K/\QQ_p$ be finite and $E/K$ a normal closure of $K.$ Let us denote $\Sigma_E:= \Hom(E,E)$ and fix one embedding $\sigma_0 \in \Sigma_E.$
	\begin{defn} \label{def:bpair}
		An \textbf{$E$-$B$-pair}  is  a pair $W=(W_e,W_{\mathrm{dR}}^+)$ consisting of a finite free $E\otimes_{\QQ_p}\fB_e$-representation $W_e$  of $G_K$ together with a $G_K$-equivariant $E \otimes_{\QQ_p}\Bdr^+$-lattice $W_{\mathrm{dR}}^+\subset W_{\mathrm{dR}}:= \Bdr \otimes_{\fB_e} W_e.$ We define the \textbf{rank} of $W$ as $\operatorname{rank}(W):=\operatorname{rank}_{E\otimes_{\QQ_p}\fB_e}W_e.$
		We denote by $C^\bullet(W)$ the complex $$[W_e \oplus W_{\mathrm{dR}}^+ \to W_{\mathrm{dR}}]$$ concentrated in degrees $[0,1].$ 
		We define the \textbf{Galois cohomology} of $W$  as 
		$$\mathbf{R}\Gamma(G_K,W):=\mathbf{R}\Gamma_{cts}(G_K,C^\bullet)$$ and write $H^i(G_K,W)$ for the cohomology groups.  
		
		We denote by $\mathcal{BP}_E$ the category of $E$-$B$-pairs with the obvious notion of morphisms. 
		A morphism $W \to W'$ of $B$-pairs is called \textbf{strict} if the co-kernel of $[W_{\mathrm{dR}}^+ \to (W')_{\mathrm{dR}}^+]$ is free. By \cite[Lemma 1.10]{nakamura2009classification} a map $f\colon W \to W'$ of $B$-pairs is strict if and only if the term wise quotient $W'/(\operatorname{im}f) = (W'_e/f(W_e),(W'_{\mathrm{dR}})^+/f(W^+_{\textrm{dR}}))$ is a $B$-pair. 
		A subobject of $W$ is called \textbf{strict } (resp. \textbf{strict at $\sigma$}) if the inclusion  is strict in the above sense (resp. strict at the component corresponding to $\sigma.$). 
	\end{defn}

	We recall for an $E$-$B$-pair $(W_e,W_{\mathrm{dR}})$ we can write (as $\Bdr^+ \otimes_{\QQ_p}E$-modules) $W_{\mathrm{dR}}^{(+)} = \prod_{\sigma \in \Sigma_E}(W_{\mathrm{dR},\tau})^{(+)}$ by using the decomposition $E \otimes_{\QQ_p} \Bdr^+ \cong \prod_{\sigma \in \Sigma_E} \Bdr^+.$
	But this decomposition is in general not $G_K$-stable (but $G_E$-stable).
	Instead we can consider the decomposition
	$W_{\mathrm{dR}}^{(+)} = \prod_{\sigma \in \Sigma_K}(W_{\mathrm{dR},\sigma})^{(+)},$
	where $W_{\mathrm{dR},\sigma}^{(+)}$ is  a module over $\Bdr^{(+)}\otimes_{K,\sigma}E.$
	This decomposition turns out to be $G_K$-stable (with respect to the trivial $G_K$-action on the right tensor factor). The $G_K$-invariants are a finite-dimensional $K\otimes_{K,\sigma}E = E$-vector space.

	\begin{defn}
		
		We say $W$ is \textbf{$\sigma$-de Rham }if $H^0(G_K,W_{\mathrm{dR},\sigma})$ is free of rank $\operatorname{rank}(W)$ over $E.$ 
		We say $W$ is \textbf{positive at $\sigma$} if the same holds already for $H^0(G_K,W^+_{\mathrm{dR},\sigma}).$ 
	
	\end{defn}
We warn about the following subtlety: If, say, a Galois representation $V \in \operatorname{Rep}_{\QQ_p}(G_K)$ is positive then $H^0(G_K,\Bdr^+\otimes_{\QQ_p}V)$ is $\dim_{\QQ_p}V$-dimensional but this does not imply that 
		the natural map $$\Bdr^+ \otimes_K H^0(G_K,\Bdr^+\otimes_{\QQ_p}V) \to \Bdr^+ \otimes_{\QQ_p}V$$ is an isomorphism. As a counterexample consider $\QQ_p(-1).$ The image of the natural map 
		$$\Bdr^+ \otimes_K (H^0(G_K,\Bdr^+ \otimes \QQ_p(-1))) \to \Bdr^+ \otimes_{\QQ_p} \QQ_p(-1)$$ is $\operatorname{Fil}^1(\Bdr^+) \otimes_{\QQ_p} \QQ_p(-1).$

	One checks that the property of being positive passes to quotients and strict subobjects (but not to subobjects in general as can be seen with $(\fB_e,t\Bdr^+)\subseteq (\fB_e,\Bdr^+)$).
	
	For the Galois cohomology of $E$-$B$-pairs we recall Tate-duality and the Euler--Poincaré formula. 
	\begin{thm} \label{thm:BpairEulerTate} Let $B_E$ be the trivial $E$-$B$-pair of rank one. Let $W$ be an $E$-$B$-pair. 
		\begin{enumerate}[(i)]
			\item We have $H^2(G_K,B_E(1)) \cong E.$
			\item $\sum_{i=0}^2(-1)^i \dim_E H^i(G_K,W)= -[K:\QQ_p]\operatorname{rank}(W).$
			\item The cup product induces a perfect pairing $$H^i(G_K,W) \times H^{2-i}(G_K,W^*(1)) \to H^2(G_K,B_E(1)) \cong E.$$
		\end{enumerate}
	\end{thm}
	\begin{proof} See \cite[Appendix]{Nak13}.
	\end{proof}
	In modern terms $B$-pairs should be viewed as an explicit description of equivariant vector bundles on the Fargues--Fontaine curve (cf.\cite{farguesfontainecourbe}). Originally Berger established an equivalence between $B$-pairs and $(\varphi,\Gamma)$-modules in \cite{berger2008construction}.
	
	\subsection{Cyclotomic $(\varphi,\Gamma)$-modules}
	For $K/\QQ_p$ finite let us recall the construction of  $\mathbf{B}^{\dagger}_{\text{rig},K},$ which is not to be confused with the Robba ring $\cR_K$ with coefficients in  $K.$ We denote by $K_0$ the maximal unramified subextension of $K.$
	Let $\tilde{\fA}:= W(\CC_p^\flat)$ and $\tilde{\fB}:=\tilde{\fA}[1/p].$
	Let $w$ be the Teichmüller lift of a pseudouniformiser of $\mathcal{O}_{\CC_p^\flat}.$
	Any element of $f \in W(\mathcal{O}_{\CC_p^\flat})[1/p,1/w]$ can be written uniquely as a convergent series $$f= \sum_{k\gg -\infty}p^k[x_k]$$ with $x_k$ a bounded sequence in $\CC_p^\flat.$ For $0<r<1$ define $$\abs{f}_r := \sup_k \abs{p}^k(\abs{x_k}_\flat)^{-\log_p(r)}$$ and $\abs{f}_I := \sup_{r \in I} \abs{f}_r.$ \footnote{The precise value $\abs{f}_r$ depends on the chosen normalisation, but an expression of the form $\abs{f}_r = \abs{x}$  for some $x \in \CC_p$ is independent of the normalisation. } For a closed interval $I \subseteq (0,1)$ we denote by $\tilde{\fB}^I$ the completion of $ W(\mathcal{O}_{\CC_p^\flat})_L[1/\pi_L,1/w]$ with respect to $\abs{-}_I.$
	Finally one defines $\tilde{\fB}^{\dagger,r}:=\varprojlim_{s<1} \tilde{\fB}^{[r,s]}.$
	Let $\mathbf{A} \subset \tilde{\fA}$ be the maximal unramified extension of the $p$-adic completion of $\mathcal{O}_{K_0}(([\varepsilon]-1))$ and $\mathbf{B} = \fA[1/p].$ Set $\mathbf{B}_K:=\mathbf{B}^{\Gal(\overline{K}/K_{K_{\text{cyc}}})}$ and $\mathbf{B}_K^{\dagger,r}:= \tilde{\fB}^{\dagger,r} \cap \fB_K.$ Lastly, define $\fB^{\dagger,r}_{\text{rig},K}$ (resp. $\fB^{\dagger}_{\text{rig},K})$ as the completion of  $\fB^{\dagger}_{K}$ with respect to the Fréchet topology (resp. as $\bigcup_r\fB^{\dagger,r}_{\text{rig},K}$ ).
	Recall that $\tilde{\fB}$ is equipped with natural actions $\varphi_p$ and $G_K$ inducing actions of $\varphi_p$ and $\Gamma^{\text{cyc}}_K:=G_K/\Gal(\overline{K}/K_{\text{cyc}}).$
	\begin{defn} Let $R$ be a topological $\mathcal{O}_{K_0}$-algebra endowed with an endomorphism $\varphi_p$ extending the Frobenius endomorphism of $K_0$ and a continuous action of $\Gamma^{\text{cyc}}_K.$
		A $(\varphi_p,\Gamma^{\text{cyc}}_K)$-module over $R$ is a finite free $R$-module $D$ with $\varphi_p$-semilinear endomorphism $\varphi$ and a continuous semilinear $\Gamma^{\text{cyc}}_K$-action commuting with $\varphi$ and such that 
		$$\varphi^*(D) =R \otimes_{\varphi_p,R}  D \xrightarrow{\id\otimes \varphi_D} D$$ is an isomorphism.
	\end{defn}
	\begin{lem}
		For a $(\varphi_p,\Gamma^{\text{cyc}}_K)$-module over $E \otimes_{\QQ_p} \fB^{\dagger}_{\text{rig},K}$ it suffices to assume freeness over $\fB^{\dagger}_{\text{rig},K},$ i.e., a not-necessarily free $(\varphi_p,\Gamma^{\text{cyc}}_K)$-module over $E \otimes_{\QQ_p} \fB^{\dagger}_{\text{rig},K}$-module, whose underlying $\fB^{\dagger}_{\text{rig},K}$-module is free is also free over $E \otimes_{\QQ_p} \fB^{\dagger}_{\text{rig},K}.$
	\end{lem}
	\begin{proof}
		See \cite[Lemma 1.30]{nakamura2009classification}.
	\end{proof}
	We recall that a rank one $(\varphi_p,\Gamma^{\text{cyc}}_K)$ over $E \otimes_{\QQ_p}\fB^{\dagger}_{\text{rig},K}$ admits a basis with respect to which the $(\varphi_p,\Gamma^{\text{cyc}}_K)$-action is given by a character 
	$\delta \colon K^\times  \to E^\times$ and we define the degree to be $v_p(\delta(\pi_K)).$ In general we define $\deg(D):= \deg(\det D),$ where $\det D$ denotes the highest exterior power of $D.$ 
	We further define the slope $\mu(D):= \deg(D)/\operatorname{rank}(D).$ 
 We call $D$ semi-stable of slope $s$ if $\mu(D)=s$ and every $\varphi$-submodule  $0 \neq D' \subseteq D$ has slope $\mu(D')\geq \mu(D).$
	\begin{thm}
		There is an exact equivalence of categories between the category of $E$-$B$-pairs and $(\varphi_p,\Gamma_K^{\text{cyc}})$-modules over $E\otimes_{\QQ_p}\fB^{\dagger}_{\text{rig},K}.$ 
		The functor $V \mapsto W(V) := (\fB_e \otimes_{\QQ_p}V,\Bdr^+\otimes_{\QQ_p},V)$ defines a fully faithful functor from the category of $E$-linear representations of $G_K$ to the category of $E$-$B$-pairs, whose essential image consists of those $B$-pairs $W$ whose $(\varphi_p,\Gamma_K^{\text{cyc}})$-module $D(W)$ is semi-stable of slope $0.$
	\end{thm}
	\begin{proof}The case $E=\QQ_p$ is \cite[Théorème 2.2.7]{berger2008construction}.
		For general $E$ see	\cite[Theorem 1.36]{nakamura2009classification}.
	\end{proof}

	\section{Partial Selmer groups}
	\subsection{Rank one $B$-pairs and twists}
	Let us recall that any rank one $E$-$B$-pair is isomorphic to $W(\delta)$ for a continuous character $\delta \colon K^\times \to E^\times.$ (cf. \cite[Theorem 1.45]{nakamura2009classification})
	Here $W(\delta):= W(\mathbf{B}^{\dagger}_{\text{rig},K}(\delta)).$
	Where $\mathbf{B}^{\dagger}_{\text{rig},K}(\delta))$ is the rank one module obtained as follows: Write $\delta = \delta_0 \delta_1$ such that $(\delta_0)_{\mid \mathcal{O}_L^\times} = (\delta)_{\mid \mathcal{O}_L^\times}$ and $\delta_0(\pi_K)=1.$ Then $\delta_0$ corresponds by local class field theory to a character $G_K \to E^\times$ (which in turn corresponds to an étale $(\varphi,\Gamma)$-module $\mathbf{B}^{\dagger}_{\text{rig},K}(\delta_0))$ and take $\mathbf{B}^{\dagger}_{\text{rig},K}(\delta_1)$ to be the rank one module with basis $e_{\delta_1}$ with trivial $\Gamma$-action and $\varphi_q(e_{\delta_1}) = \delta_1(\pi_K)e_{\delta_1}.$ Finally set 
	$\mathbf{B}^{\dagger}_{\text{rig},K}(\delta):= \mathbf{B}^{\dagger}_{\text{rig},K}(\delta_0) \otimes \mathbf{B}^{\dagger}_{\text{rig},K}(\delta_1).$
	
	\begin{lem} \label{lem:deltatwist}
		Let $\delta \colon K^\times \to E^\times$ be of the form 
		$\delta = \prod_{\sigma \in \Sigma_K} \sigma(x)^{k_\sigma}.$ With $k_\sigma \in \ZZ.$ Then 
		$$W(\delta) = (E \otimes_{\QQ_p} \fB_e, \bigoplus_{\sigma\colon K \to E} t^{k_\sigma} \Bdr^+\otimes_{K,\sigma} E).$$
		In particular, twisting with $\delta$ does not change the $\fB_e$-component of a $B$-pair. 
	\end{lem}
	\begin{proof}
		\cite[Lemma 2.12]{nakamura2009classification}.
	\end{proof}
	\subsection{Ding's partial de Rham $B$-pairs}
	
	Ding defined for $J\subset \Sigma_K$
	$$H^1_{g,J}(G_K,W) := \ker[H^1(G_K,W) \to \prod_{\sigma \in J} H^1(G_K,W_{\mathrm{dR},\sigma})].$$
	
	Let us recall another construction from \cite{ding2017protect}. 
	\begin{defn}
		For $\kappa = (\kappa_{\sigma})_{\sigma \in \Sigma_K} \in \ZZ^{\Sigma_K}$ we can define a character $\delta_{\kappa} \colon K^\times \to E^\times$ by $x \mapsto \prod \sigma(x)^{\kappa_\sigma}.$ For an $E$-$B$-pair $W=(W_e,W_{\mathrm{dR}}^+)$ we denote by $W(\delta_{\kappa})$ the $E$-$B$-pair with $W(\delta_{\kappa})_e:=W_e$ and $W^+(\delta_{\kappa})_{\mathrm{dR},\sigma}= t^{\kappa_{\sigma}}W^+_{\mathrm{dR},\sigma}.$ 
	\end{defn}
	\begin{lem}
		\label{lem:sesWdelta}
		Let $J\subset \Sigma_K$ and $W$ be $\sigma$-de Rham for all $\sigma \in J.$ For each $\sigma \in J$ choose $\kappa_\sigma>0$ such that $H^0(G_K,t^{\kappa_{\sigma}}W_{\mathrm{dR},\sigma}^+)=0$ then we have a short exact sequence of complexes 
		$$0 \to C^\bullet(W(\delta_{\kappa})) \to C^\bullet(W)\to \bigoplus_{\sigma \in J} W_{\mathrm{dR},\sigma}^+/t^{\kappa_\sigma}W_{\mathrm{dR},\sigma}^+[0] \to 0$$
		inducing a short exact sequence 
		$$0 \to H^0(G_K,W) \to \oplus_{\sigma \in J} H^0(G_K,W^+_{\mathrm{dR},\sigma}) \to H^1(G_K,W(\delta)) \to H^1_{g,J}(G_K,W)\to 0$$
	\end{lem}
	\begin{proof}
		See \cite[after Lemma 1.11]{ding2017protect}
	\end{proof}
	\subsection{Comparison to the analytic case}
	
	\begin{defn}\label{def:analytic}
		We say $ V \in \operatorname{Rep}_E(G_K)$ is $K$-analytic (with respect to $\sigma_0 \in \Sigma_K$) if $V$ is Hodge--Tate of weight $0$ at every $\sigma \neq \sigma_0.$ By this we mean
		$\CC_p \otimes_{K,\sigma} V \cong \CC_p^{\dim_KV}$ for every $\sigma \neq \sigma_0.$ Equivalently, $H^0(G_K,\CC_p \otimes_{K,\sigma} V)$ is of dimension $\dim_KV.$ 
		We say an $E$-$B$-pair $W$ is $K$-analytic if the representation 
		$W_{\text{dR},\sigma}^+/t W_{\text{dR},\sigma}^+$ is the trivial $\CC_p$-semilinear representation for every $\sigma \neq \sigma_0.$ Equivalently, if $H^0(G_K,W_{\text{dR},\sigma}^+/t W_{\text{dR},\sigma}^+)$ is of dimension $\operatorname{rank}(W).$ 
	\end{defn}
	It is easy to see that a representation $V$ is $K$-analytic if and only if the $B$-pair attached to $V$ is $K$-analytic.
	For a subset $\Sigma \subset \Sigma_K$ and $V \in \operatorname{Rep}_E(G_K)$ we define 
	$$H^1_{\Sigma}(V):= \operatorname{Ker}[H^1(G_K,V) \to \prod_{\sigma \in \Sigma}(H^1(G_K,B_{HT} \otimes_{K,\sigma} V))].$$
	We denote by $\nabla_\sigma\colon H^1(G_K,V) \to H^1(G_K,B_{HT} \otimes_{K,\sigma} V)$ the natural map induced by $v \mapsto 1 \otimes v.$
	\begin{prop}Suppose $V$ is Hodge--Tate at $\sigma.$ Then $X \in H^1(G_K,V)$ (viewed as an extension of $E$ by $V$) is Hodge--Tate at $\sigma$ if and only if $\nabla_{\sigma}X=0.$
		If $V$ is $\CC_p$-admissible at $\sigma,$ then $X$ is $\CC_p$-admissible at $\sigma$ if and only if $\nabla_\sigma X=0.$
	\end{prop}
	\begin{proof}
		Suppose $X$ is Hodge--Tate at $\sigma.$  But then $B_{HT} \otimes_{K,\sigma}X \cong B_{HT}^{\dim_KV+[E:K]}$ and in particular $\nabla_\sigma X=0.$
		If $\nabla_\sigma X=0,$ then $B_{HT}\otimes_{K,\sigma}X = B_{HT} \otimes_{K,\sigma}V \oplus B_{HT} \cong B_{HT}^{\dim_K V+[E:K]}$ as $V$ is assumed to be Hodge--Tate. 
		The second part follows along the same lines using that $\CC_p$-admissibility is equivalent to being Hodge--Tate of weight $0.$
	\end{proof}
	\begin{lem}\label{lem:htW0derham}
		Let $W$ be a free $\Bdr^+$-module of rank $d$ with a continuous semi-linear action of $G_K$ such that $W/tW \cong \CC_p^d$ as representations.
		Then 
		$$\Bdr^+ \otimes_KH^0(G_K,W) \to W$$ is an isomorphism.
	\end{lem}
	\begin{proof}
		First consider the short exact sequence
		$$0 \to tW/t^2W \to W/t^2W \to W/tW \to 0.$$
		We have by assumption $tW/t^2W \cong C_p(1)^d$ and hence by \cite{iovita1999galois} $H^i(G_K,tW/t^2W)=0$ for $i=0,1$ which implies 
		$H^0(G_K,W/t^2W) \cong H^0(G_K,W/tW) (\cong K^d).$
		Arguing by induction we conclude 
		$H^0(G_K,W) \cong H^0(G_K,W/tW).$
		Tensoring with $\CC_p = \Bdr^+/t\Bdr^+$ we get
		$$\CC_p \otimes_K H^0(G_K,W)  \cong \CC_p \otimes_K H^0(G_K,W/tW) = W/tW,$$ by the assumption $W/tW \cong \C_p^d.$

The natural map 
	$$\Bdr^+ \otimes_KH^0(G_K,W) \to W$$  is thus a map between finitely generated $\Bdr^+$-modules which is an isomorphism modulo the maximal ideal $t\Bdr^+,$ which by Nakayama's Lemma implies that the map itself is an isomorphism.
	
	\end{proof}
	%
	
	\begin{lem}
		\label{lem:deRhamselmer}
		Let $V$ be de Rham at $\sigma \colon K \to E$ of Hodge--Tate weight $0.$ Then $X \in H^1(G_K,V)$ is de Rham at $\sigma$ if and only if $X$ is Hodge--Tate at $\sigma.$
		In particular, the kernel of the natural map $\nabla_\sigma\colon H^1(G_K,V) \to H^1(G_K,B_{HT} \otimes_{K,\sigma} V))$  is the geometric Bloch--Kato group 
		$$H^1_{g,\sigma}(V)= \operatorname{ker}(H^1(G_K,V) \to H^1(G_K,\Bdr\otimes_{K,\sigma}V))$$
		at $\sigma.$
	\end{lem}
	\begin{proof}
		
		Since de Rham implies Hodge--Tate one implication is trivial. Let us assume that $X$ is Hodge--Tate at $\sigma.$ For symmetry reasons we can assume $\sigma=\id.$ First observe, that the only Hodge--Tate weight of $X$ is $0.$ Indeed, $\Hom_{\CC_p}(\CC_p,\CC_p(d)) \cong \CC_p(d)$ and recall $H^0(G_K,\CC_p(d))=0$ for $d\neq 0$ (cf. \cite[Corollary 3.56]{Fontaine}). If $X$ had a Hodge--Tate weight $d \neq 0,$ then we could write $\CC_p \otimes_{K} X= (\CC_p(d) \otimes_{K}E)  \oplus W$  and $\operatorname{im}(\CC_p \otimes_K V \to \CC_p \otimes_K X) \subset W$ by the above reasoning. But $\CC_p \otimes_K X/V \cong \CC_p \otimes_K E \ncong \CC_p(d) \otimes_K E.$ Now we can apply Lemma \ref{lem:htW0derham} to $\Bdr^+ \otimes_K X$ to conclude $\dim_KH^0(G_K,\Bdr \otimes_K X) \geq \dim_KH^0(G_K,\Bdr^+ \otimes_K X)= \dim_K(H^0(G_K,\CC_p\otimes_K X))=\dim_K(X),$ which means that $X$ is de Rham at $\sigma=\id.$
	\end{proof}
\begin{prop}
	Let $W$ be a $K$-analytic $B$-pair. Then an extension $X \in H^1(G_K,W)$ belongs to $H^1_{g,\Sigma \setminus\{\sigma_0\}}(G_K,W)$ if and only if $X$ is $K$-analytic.
\end{prop}
\begin{proof}
	It is clear that extensions belonging to $H^1_{g,\Sigma \setminus \{\sigma_0\}}(G_K,W)$ are de Rham at $\Sigma_K \in \setminus \{\sigma_0\}$ and by additivity of weights they are of Hodge--Tate weight $0,$ which in particular implies that they are $K$-analytic. 
	By Lemma \ref{lem:htW0derham} any $K$-analytic $B$-pair is de Rham at every $\sigma \in \Sigma_K \setminus \{\sigma_0\}$ hence the claim.
\end{proof}
	\section{Selmer groups as delta functors}
	\subsection{Euler--Poincaré formula}
	Let us abbreviate $H^i(W):= H^i(G_K,W)$ (similarly for $H^1_{g,\Sigma}).$
	\begin{lem}
		\label{lem:EPCgeneral}
		Suppose $W$ is positive for all $\sigma \in J \subset \Sigma_K$ and choose $\delta=\delta_{\kappa}$ as in Lemma \ref{lem:sesWdelta}. 
		Then 
		$$-([K:\QQ_p]-\# J)r  = \dim_EH^0(W)  -\dim_EH^1_{g,J}(W) + \dim_E H^2(W(\delta)) .$$
		
	\end{lem}
	\begin{proof}
		By positivity the second term in the short exact sequence of  Lemma \ref{lem:sesWdelta} is isomorphic to $[E^{r}]^{J}$ and according to the Euler--Poincaré Formula (cf. Theorem \ref{thm:BpairEulerTate}) the dimension of $H^1(W(\delta))$ is given by $$ [K:\QQ_p]r+\dim_E(H^0(W(\delta)))+\dim_EH^2(W(\delta)).$$ The character $\delta$ is chosen in a way which ensures the vanishing of $H^0(W(\delta)).$ 
		Furthermore from the exact sequence
		\begin{align*}\dim_E H^1_{g,J}&=\dim_E H^1(W(\delta))-[\#J r-\dim_E H^0(W)]\\
			&= ([K:\QQ_p]-\#J)r + \dim_E H^0(W)+\dim_E H^2(W(\delta)).\end{align*}
	\end{proof}
	
	\begin{cor}	Let $0 \neq W$  be positive at all $\sigma \in \Sigma_K.$ Then for any exhaustion 
		$$J_0=\emptyset \subsetneq J_1 \dots \subsetneq J_d= \Sigma_K$$ 
		we get a flag
		$$H^1(W)\supsetneq H^1_{g,J_1}(W)\supsetneq \dots \supsetneq H^1_{g}(W).$$
	\end{cor}
	\begin{proof}
		The chain $H^1_{g,J_i}(W)$ is indeed strictly descending. To see this note that by  Lemma \ref{lem:EPCgeneral} the dimension of each graded piece $H^1_{g,J_i}(W)/H^1_{g,J_{i+1}}(W)$ is $\operatorname{rank}(W) \neq 0.$
	\end{proof}
	
	Observe that the character $$\operatorname{Norm}_{K/\QQ_p}(-)\abs{\operatorname{Norm}_{K/\QQ_p}(-)}_{p} \colon K^\times \to E^\times$$ corresponds by local class field theory to the cyclotomic character $\chi_{cyc}$ (with $\abs{p}_p = 1/p$).
	
	\begin{defn} For an embedding $\sigma$ we define the character $\chi_{C,\sigma}:= \sigma(x)\abs{x},$ where the absolute value is normalised as $\abs{x} := \abs{\operatorname{Norm}_{K/\QQ_p}(x)}_p$ and $\abs{p}_p = 1/p,$ which results in 
		$\abs{\pi_K} = \abs{p^f}_p = \frac{1}{p^f},$ where $f$ is such that $\abs{\mathcal{O}_K/(\pi_K)} = p^f$
	
		 We simply write $\chi_{C}$ for $\chi_{C,\id}.$ 
	\end{defn}
	For $K=\QQ_p$ we have $\chi_{C} = \chi_{cyc}.$ The slope of $\chi_{C,\sigma}$ is $\operatorname{val}_{\pi_K}(\sigma(\pi_K)\abs{\pi_K}) = \operatorname{val}_{\pi_K}(\pi_K/p^f) = -([K:\QQ_p]-1).$%
	\begin{prop}(Euler--Poincaré formula for analytic cohomology)
		Suppose $W$ is de Rham of weight $0$ for all $\sigma \in \Sigma_K \setminus\{\id\}.$ Then in the situation of Lemma \ref{lem:EPCgeneral} we have 
		$\dim_E H^2(W(\delta)) = \dim_E H^0(W^*(\chi_{C}))$ and 
		$$	-r = \dim_E H^0(W) -\dim_E H^1_{g,\Sigma_K \setminus\{\id\}}(W) + \dim_E H^0(W^*(\chi_{C})).$$
	\end{prop}
	\begin{proof}
		We can take $\delta = (0,1,\dots,1)$ with the $0$ at $\sigma=\id$ and thus $W(\delta)^*(\chi_{cyc}) = W^*(x/\abs{x}) = W^*(\chi_{C})$ and by Tate duality $$\dim H^2(W(\delta)) = \dim H^0(W(\delta)^*(1)) = \dim H^0(W^*(\chi_{C})).$$
	\end{proof}
	
	\subsection{Main results}
	The calculations of Lemma \ref{lem:EPCgeneral} hint at the fact that the groups $H^0(W), H^1_{g,J}(W), H^2(W(\delta))$ give a reasonable cohomology theory for $B$-pairs with non-positive\footnote{Confusingly, a representation with non-positive Hodge--Tate weights is called positive (because the de Rham filtration jumps are in positive degrees).} Hodge--Tate weights at $\sigma \in J$ and that $\chi_D=\chi_{cyc}/\delta$ could be a reasonable ``dualising object''. We will make this precise below. At first glance it seems surprising that we have (certain) liberties with respect to the choice of $\delta.$ This is (essentially) due to the following Lemma:
	\begin{lem} \label{lem:nak}
		Let $W=B(\rho)$ be the rank one $B$-pair attached to $\rho \colon K^\times \to E^\times.$ Then $H^0(W)\cong E$ if and only if $\rho = \prod_{\sigma}\sigma^{\kappa_\sigma}$ with all $\kappa_\sigma \leq 0.$ Otherwise $H^0(W)=0.$
	\end{lem}
	\begin{proof}
		See \cite[Proposition 2.14]{nakamura2009classification}.
	\end{proof}
	We would like to define $H^2_{g,J}(W):=H^2(W(\delta)).$
	By construction we will have $H^2_{g,J}(B(\chi_D))= H^2(B(\chi_D\delta)) = H^2(B(\chi_{cyc})) \cong E$. Which will make $\chi_D$ our dualising object. 
	But $\delta$ has to be chosen depending on the Hodge--Tate weights of $W.$ 
	By duality from Lemma \ref{lem:nak} we deduce:
	\begin{rem} \label{rem:duality}
	Consider $\delta_1 = \prod_{\sigma\in J} \sigma^{\kappa_{\sigma,1}}$ and $\delta_2 = \prod_{\sigma\in J} \sigma^{\kappa_{\sigma,2}}.$ With $\kappa_{\sigma,2} \geq\kappa_{\sigma,1}>0.$ Let $\chi_{D_i}:=\chi_{cyc}/\delta_i$ then
	 $H^2(B(\chi_{D_1} \delta_2)) =H^2(B(\chi_{cyc}) \frac{\delta_2}{\delta_1}) \cong E.$
	\end{rem}

	\begin{lem}
		\label{lem:Hodgetateacyclic}
		Let $X$ be a finitely generated $\Bdr^+$-representation of $G_K,$ for $j \in \NN_0$ let $F^j := t^{j}X$ and suppose that the $\CC_p$-representations $F^j/F^{j+1}$ are Hodge--Tate and do not admit $0$ as a Hodge--Tate weight. Then 
		$$H^n(G_K,X)=0$$
		for every $n\in \NN_0.$
	\end{lem}
	\begin{proof}
		If $i\neq 0$ then by \cite[Proposition 3.2]{iovita1999galois} we have $H^n(G_K,\CC_p(i))= 0$ for all $n \in \NN_0$ and thus 
		$H^n(G_K,X/F^1) = 0$ for all $n.$
		Using the short exact sequences 
		$$0 \to F^j/F^{j+1}\to X/F^{j+1} \to X/F^j \to 0$$
		we deduce by induction on $j \geq 1$ that $H^n(G_K,X/F^{j})=0$ for every $j.$
		Passage to the limit yields the claim. 
	\end{proof}

	The category of $B$-pairs is not abelian, it is however an exact category (with the obvious class of short exact sequences). We subsequently use the word $\delta$-functor to mean a family of functors $H^n(-) \colon\mathcal{E} \to \mathcal{A}$ indexed by $n \in \NN_0$ from an exact category $\mathcal{E}$ into an abelian category $\mathcal{A}$, taking short exact sequences to long exact sequences in cohomology and such that the connecting homomorphisms are natural.

	\begin{thm} \label{thm:deltafunktor}
		Let $0 \to W_1 \to W_2 \to W_3 \to 0$ be a short exact sequence of $E$-$B$-pairs. 
		Suppose each $W_i$ is positive with Hodge--Tate weights in $[-n_{\sigma},0]$ at all $\sigma \in J$ and let $\delta = \prod_{\sigma \in J} \sigma^{\kappa_\sigma}$ with $\kappa_{\sigma}\geq n_{\sigma}.$ 
		Let $\partial^1_{W_3(\delta)}:H^1(W_3(\delta))\to H^2(W_1(\delta))$ be the connecting homomorphism. 
	
		Then 
		\begin{enumerate}[(i)]
			\item The connecting homomorphism 
			$\partial^0 \colon H^0(W_3) \to H^1(W_1)$ takes values in $H^1_{g,J}(W_1).$
			\item The map 
			$$\partial^1 \colon H^1_{g,J}(W_3) \to H^2(W_1(\delta)),$$ given by $\partial^1(x):= \partial_{W(\delta)}^1(\tilde{x}),$ where $\tilde{x} \in H^1(W_3(\delta))$ is any preimage under the projection from Lemma \ref{lem:sesWdelta} is well-defined.
			\item  Setting $H^0_{g,J}(W)=H^0(W)$ and $H^2_{g,J}(W):=H^2(W(\delta))$ (and $H^i_{g,J}(W)=0$ for $i\geq 3$) we obtain a $\delta$-functor from the category of $J$-de Rham $B$-pairs with Hodge--Tate weights in $([-n_\sigma,0])_{\sigma}$ to $E$-vector spaces.
		\end{enumerate}
	\end{thm}
	\begin{proof}
		For the first part, let $x \in H^0(W_3).$ This defines a morphism $x \colon B_E \to W_3.$ Pulling back the image of $x$ to $W_2$ defines an extension $0 \to W_1 \to X \to B_E \to 0$ which is equal to $\partial^0(x) \in H^1(W_1).$  Being a sub-object of $W_2$ it is $\sigma$-de Rham at all $\sigma \in J,$ i.e., it belongs to $H^1_{g,J}(W_1).$
		By Lemma \ref{lem:sesWdelta} we have exact sequences of complexes 
		\begin{equation} \label{eq:shortying}
			0 \to C(W_i(\delta)) \to C(W_i) \to \prod_{\sigma \in J} {W_{i,\mathrm{dR},\sigma}}^+/t^{k_\sigma} {W_{i,\mathrm{dR},\sigma}}^+[0] \to 0.
		\end{equation}
		Consider the piece of the long exact sequence attached to the short exact sequence \eqref{eq:shortying} 
		
		\begin{equation}
			\label{eq:sequenceying} \prod H^0( W_{i,\mathrm{dR},\sigma}^+/t^{k_\sigma}W_{i,\mathrm{dR},\sigma}^+)  \xrightarrow{f_i} H^1(W_i(\delta))\xrightarrow{g_i}  H^1(W_i) \to \prod H^1( W_{i,\mathrm{dR},\sigma}^+/t^{k_\sigma}W_{i,\mathrm{dR},\sigma}^+)\end{equation}
		and the commutative diagram 
		\begin{equation} \label{eq:dellcd}	\begin{tikzcd}
				{H^0(W_{2,\mathrm{dR},\sigma}^+/t^{k_\sigma})} & {H^0(W_{3,\mathrm{dR},\sigma}^+/t^{k_\sigma})} & 0 \\
				{H^1(W_2(\delta))} & {H^1(W_3(\delta))} \\
				\arrow[ from=1-1, to=1-2]
				\arrow["{f_2}", from=1-1, to=2-1]
				\arrow[from=1-2, to=1-3]
				\arrow["{f_3}", from=1-2, to=2-2]
				\arrow[, from=2-1, to=2-2]
			\end{tikzcd}\
		\end{equation}
		The surjectivity of the boundary map in the top row of \eqref{eq:dellcd} follows by applying Lemma \ref{lem:Hodgetateacyclic} to see $H^0( W_{i,\mathrm{dR},\sigma}^+/t^{k_\sigma}W_{i,\mathrm{dR},\sigma}^+) = H^0( W_{i,\mathrm{dR},\sigma}^+)$ and the surjectivity of $H^0( W_{2,\mathrm{dR},\sigma}^+)\to H^0( W_{3,\mathrm{dR},\sigma}^+),$ which is a consequence of the positivity hypothesis. 
		We prove that $\partial^1 \colon H^1_{g,J}(W_3) \to H^2(W_1(\delta))$ is well-defined. 
		Suppose $y \in H^1(W_3(\delta))$ is mapped to $0$ in $H^1(W_3)$ then it belongs to the image of $f_3$ but by \eqref{eq:dellcd} thus to the image of $H^1(W_2(\delta)),$ which lies in the kernel of $H^1(W_3(\delta)) \to H^2(W_1(\delta)).$ This shows that $\partial^1$ is well-defined. 
		To see exactness at the connecting homomorphism at $H^1_{g,J}(W_3) \to H^2(W_1(\delta)) \to H^2(W_2(\delta))$
		note that $\operatorname{Im}(\partial) \subset \ker(H^2(W_1(\delta))\to H^2(W_2(\delta))).$
		Suppose $x \in \ker(H^2(W_1(\delta))\to H^2(W_2(\delta))).$
		Then from the usual exact sequence of Galois cohomology for the $W_i(\delta)$ we can find $y \in H^1(W_3(\delta))$ such that $x= \partial y$ by the exactness of \eqref{eq:shortying} $g_3(y) $ belongs to $H^1_{g,J}(W_3)$ and, by construction, $\partial^1(g_3(y))=x$ thus proving 
		$\operatorname{Im}(H^1_{g,J}(W_3) \to H^2(W_1(\delta)))= \ker(H^2(W_1(\delta))) \to H^2(W_2(\delta)).$
		The exactness in degree $2$ follows from the fact that $W \mapsto (H^i(W(\delta)))_i$ is a delta functor. 
		The exactness in degree $0$ follows from the fact that $W \mapsto (H^i(W))_i$ is a delta functor. 
		The exactness at $H^0(W_2)\to H^0(W_3) \to H^1_{g,J}(W_1)$ is clear. 
		It remains to check exactness at $H^1_{g,J}(W_1)\to H^1_{g,J}(W_2) \to H^1_{g,J}(W_3).$ Suppose $X \in H^1_{g,J}(W_2)$ is an extension. First of all note that $X/W_1$ indeed belongs to $H^1_{g,J}(W_3)$ because the positivity is inherited by quotients. If $X/W_1$ is split, i.e, $X$ is mapped to zero in $H^1_{g,J}(W_3)$ then we find a preimage in $H^1(W_1)$ more precisely, the preimage in $W_2$ of the image of the section $B_E \to X/W_3$  to $X$ defines an extension $0 \to W_1 \to Z \to B_E \to 0$ which, by construction, is a strict sub object of $X$ and hence belongs to $H^1_{g,J}(W_1).$
	\end{proof}
	\begin{lem} 
		\label{lem:deltamodification}
		Let $W$ be de Rham at $\sigma$ with Hodge--Tate weights in $[-n,0].$ For $k \geq n+1$ define $\delta_k:= \sigma^{k}.$
		Then the natural maps
		$H^i(W(\delta_k)) \to H^i(W(\delta_{n+1}))$ are isomorphisms for every $i.$ 
	\end{lem}
	\begin{proof}
		By Lemma \ref{lem:sesWdelta} we get a short exact sequence of complexes 
		$$0 \to C^\bullet(W(\delta_k)) \to C^\bullet(W(\delta_{n+1})) \to t^{n+1}W_{\mathrm{dR},\sigma}^+/t^{k}W_{\mathrm{dR},\sigma}^+[0] \to 0.$$
		We claim that $H^i(G_K,t^{n+1}W_{\mathrm{dR},\sigma}^+/t^{k}W_{\mathrm{dR},\sigma}^+)=0.$
		Let us write $W^+:=W_{\mathrm{dR},\sigma}^+.$
		By the assumption on the weights $$W_{\mathrm{dR},\sigma} = \Bdr \otimes_K H^0(G_K,W_{\mathrm{dR},\sigma}^+)$$  and we have
		$$\Bdr^+ \otimes_KH^0(G_K,W^+) \subseteq W^+ \subseteq t^{-n}\Bdr^+ \otimes_KH^0(G_K,W^+).$$
		Hence $t^{n+1}W^+$ is a subrepresentation of $(t\Bdr^+)^d$ for some $d\geq0.$
		We get a short exact sequence of finite length $\Bdr^+$-representations
		$$0 \to t^{n+1}W^+/t^kW^+ \to (t\Bdr^+/t^{k+1}\Bdr^+)^d \to Q \to 0,$$ where the last term is defined as the cokernel of the first non-trivial map.
		For $j \geq 1$ we have $H^i(G_K,t^j\Bdr^+)=0$ for  every $i$ from which we deduce that $H^0(G,  t^{n+1}W_{\mathrm{dR}}^+/t^{k}W_{\mathrm{dR}}^+)=0$ holds. 
		Now apply \cite[Proposition 3.1]{Fontaine2003} to first conclude that $H^1(G_K,t^{n+1}W_{\mathrm{dR}}^+/t^{k}W_{\mathrm{dR}}^+)=0,$ as a consequence $H^0(G_K,Q)=0,$ applying the same result again yields $H^1(G_K,Q)=0,$ which finally implies $H^2(G_K,t^{n+1}W^+/t^kW^+)=0$ as desired. The vanishing in degrees beyond $2$ is clear. 
	\end{proof}
	We can now clarify the dependence on $\delta$ for the second cohomology group $H^2_{g,J}(W).$
	\begin{cor} \label{cor:indep}
		Let $W$ be a $B$-pair which at $\sigma \in J$ is positive with weights in $[-n_\sigma,0].$ Let $\delta_1 = \prod_{\sigma\in J} \sigma^{\kappa_{\sigma,1}}$ and $\delta_2 = \prod_{\sigma\in J} \sigma^{\kappa_{\sigma,2}}$ with $\kappa_{\sigma,2} \geq\kappa_{\sigma,1}>n_\sigma.$
		Let $H^i_{g,J,\delta_j}(W)$ be the group $H^i_{g,J}$ in Theorem \ref{thm:deltafunktor} defined using $\delta_j$.

		\begin{enumerate}[(i)]
			\item The groups $H^i_{g,J,\delta_j}(W)$ are equal (as subsets of $H^i(W)$) for $i=0,1.$
			\item For $i=2$ the natural map $$H^2_{g,J,\delta_2}(W) = H^2(W(\delta_2)) \to H^2(W(\delta_1)) = H^2_{g,J,\delta_1}(W)$$ is an isomorphism.
		\end{enumerate}
\begin{proof}
	The independence of $H^0_{g,J}$ and $H^1_{g,J}$ is clear by definition.
	Apply Lemma	\ref{lem:deltamodification} embedding-by-embedding for the second part.
\end{proof}

	\end{cor}
	This justifies omitting $\delta$ in the definition of $H^2_{g,J}(W).$ 
	For the main result, including a duality statement, we need to ensure that both $W$ and $W^*(\chi_D)= W^*(\chi_{cyc}/\delta)$ satisfy the same bounds on the Hodge--Tate weights. We will thus work with a fixed choice of $\delta.$  
	Using Remark \ref{rem:duality} and Corollary \ref{cor:indep}
	one can compare the dualities for different choices of $\delta.$

	\begin{thm} 
		\label{thm:mainresult} Let $K/\QQ_p$ be finite, let $E$ be a normal closure of $K,$ let $J\subseteq \Sigma_K= \Hom_{\QQ_p}(K,E)$ and let $\mathbf{n} \in \NN^J.$
		Let $\mathcal{BP}_E^{J,\mathbf{n}}$ be the full subcategory of the category of $E$-$B$-pairs $W$ with the  property that
		$W$ is de Rham with Hodge--Tate weights in $[-n_{\sigma},0]$ for all $\sigma \in J.$ Let  $\delta = [x \mapsto \prod_{\sigma \in J} \sigma(x)^{n_\sigma+1}]$ viewed as a rank one object of $\mathcal{BP}_E.$ 
		Let $\chi_D:=\chi_{cyc}/\delta.$
		\begin{enumerate}[(i)]
			\item The group $H^1_{g,J}(W)$ agrees with $\operatorname{Ext}_{\mathcal{B}^{J,\mathbf{n}}_E}(B_E,W).$
			\item $(H^n_{g,J}(W))_{n \in \NN}$ is a delta functor. 
			\item $H^2_{g,J}(B(\chi_D))= H^2(B(\chi_{cyc}))\cong E$ 
			\item We have the following Euler--Poincaré formula:
			$$-\operatorname{rank}(W)([K:\QQ_p]-\abs{J})=\sum_{i=0}^2 (-1)^i \dim_EH^i_{g,J}(W).$$ 
			\item For every $W \in \mathcal{BP}_E^{J,\mathbf{n}}$ we have $W^*(\chi_D) \in \mathcal{BP}_E^{J,\mathbf{n}}$ and there exists a canonical perfect pairing
			$$H^i_{g,J}(W) \times H^{2-i}_{g,J}(W^*(\chi_D)) \to  E.$$ 
			
			\item For $J=\emptyset$ this specialises to Galois cohomology of $B$-pairs.  

		\end{enumerate}
	\end{thm}
	\begin{proof}
		The first point is standard, the second point was established in Theorem \ref{thm:deltafunktor}. The third point follows by  definition $$E\cong H^2(B(\chi_{cyc}))=H^2(B(\chi_{D}\delta)).$$ The Euler--Poincaré formula was established in Lemma \ref{lem:EPCgeneral}.\\
		For duality \footnote{We remark that the present duality for $W=B_E(\chi_D)$ has been established by Ding in \cite[Lemma 1.19]{ding2017protect}.}
		if $i=0$  we have (by definition) $H^2_{g,J}(W^*(\chi_D))= H^2(W^*(1)),$ hence we get a perfect pairing from the usual Tate-duality. The case $i=2$ is treated analogously. The non-trivial case to consider is $i=1.$ 
		Recall that $W(\delta) \to W$  by Lemma  \ref{lem:sesWdelta} defines a morphism of $B$-pairs which on cohomology induces 
		$$g_W \colon H^1(W(\delta)) \to H^1(W)$$ with image $H^1_{g,J}(W).$
		For $W^*(\chi_D)$ we get $W^*(\chi_D(\delta)) =W^*(1)$ and hence 
		$g_{W^*(\chi_D)} \colon  H^1(W^*(1)) \to H^1(W^*(\chi_D)).$
		
		Consider the following diagram of pairings:

				\begin{equation}\begin{tikzcd}
				\label{eq:pairingdiagram}
				{H^1(W(\delta))} & \times & {H^1(W^*(\chi_{D}))} && {H^2(B(\chi_{cyc}))} \\
				{H^1(W(\delta))} & \times & {H^1(W^*(1))} && {H^2(B(\chi_{cyc}\delta))} \\
				{H^1(W)} & \times & H^1(W^*(1)) && {H^2(B(\chi_{cyc})).}
				\arrow["\id",equal, from=1-1, to=2-1]
				\arrow[from=1-3, to=1-5]
				\arrow["\cong", no head, from=1-5, to=2-5]
				\arrow["g_W",, from=2-1, to=3-1]
				\arrow["g_{W^*(\chi_D)}"', , from=2-3, to=1-3]
				\arrow[from=2-3, to=2-5]
				\arrow["\id"', equal, from=3-3, to=2-3]
				\arrow[from=3-3, to=3-5]
				\arrow["\cong"', no head, from=3-5, to=2-5]
		\end{tikzcd}\end{equation}
		By functoriality of cup products the diagram is commutative. The right hand terms are all isomorphic to $E$ by Lemma \ref{lem:nak}. 
		Hence we get a well-defined pairing 
		$$\langle \cdot,  \cdot \rangle \colon  H^1_{g,J}(W) \times H^1_{g,J}(W^*(\chi_D))  \to E$$ by choosing preimages in the middle level of \eqref{eq:pairingdiagram} and evaluating in $$H^2(B(\chi_{cyc}\delta)) \cong E.$$ 
		 By the Euler--Poincaré formula the dimensions of both spaces agree, hence it suffices to show that the induced map $H^1_{g,J}(W) \to H^1_{g,J}(W^*(\chi_D))^*$ is injective. Indeed we have $$\dim H^0_{g,J}(W) = \dim H^2(W^*(\chi_{cyc})) =\dim H^2_{g,J}(W^*(\chi_D))$$ since $\chi_D= \chi_{cyc}/\delta$ and by analogous reasoning $$\dim H^0(W^*(\chi_D)) = \dim H^2(W(\chi_{D}^{-1}\chi_{cyc}))=H^2_{g,J}(W).$$ Let $x \in H^1_{g,J}(W)$ such that $\langle x,y \rangle=0$ for all $y\in H^1_{g,J}(W^*(\chi_D)).$ Viewing $x$ as an element of $H^1(W)$ we conclude from \eqref{eq:pairingdiagram} that $\langle x,y'\rangle=0$ for any $y' \in g_{{W^*(\chi_D)}}^{-1}(H^1_{g,J}(W^*(\chi_D))) = H^1(W^*(\chi_{cyc}))$ using that $$H^1(W^*(\chi_{cyc})) \xrightarrow{g_{W^*(\chi_D)}} H^1_{g,J}(W^*(\chi_D))$$ is surjective and hence $x=0$ by usual Tate--Duality.  \\
		The point \textit{(vi)} is clear if $J=\emptyset$ then $\delta$ is the trivial character and $H^i_{g,\emptyset}(W)=H^i(W).$
	
	\end{proof}
	By applying our results in the case $K=\QQ_p,$ we obtain an explicit formula for the dimension of $H^1_g.$
	\begin{cor} \label{cor:newformula}
		Let $K=\QQ_p$ and let $W$ be a de Rham $B$-pair with non-positive Hodge--Tate weights contained in $[-n,0].$  Then 
		\begin{align*}\dim_{E} H^1_g(W) &= \dim_{E} H^0(W)+  \dim_{E}H^2(W(x^{n+1})) \\
			&= \dim_{E} H^0(W)+\dim_{E} H^0(W^*(x^{-n-1})(1)). \end{align*}
	\end{cor}
	\begin{proof}
		Apply Theorem \ref{thm:mainresult} to $J = \{\id\}.$
	\end{proof}
	
	\begin{prop} \label{prop:classicformula} Let $V =V(W) \in \operatorname{Rep}_{E}(G_{K})$ be positive de Rham with Hodge--Tate weights contained in $[-n,0].$ Then
		$$\dim_{E}H^2(W(x^{n+1})) = \dim_E(D_{cris}(V^*(1)))^{\varphi=1}, $$
		where $D_{cris}(-):= H^0(G_K,\Bcris \otimes_{\QQ_p}-).$
	\end{prop}
	\begin{proof} By \cite[Proposition 3.8]{bloch1990functions} $H^1_g(V)$ is orthogonal to $H^1_{e}(V^*(1))$. Let us abbreviate $\dim:=\dim_{E}$ and write $t_V:= D_{\mathrm{dR}}(V)/D_{\mathrm{dR}}^+(V)$ and $h^i_{V}:= \dim H^i(G_K,V).$
		From the fundamental exact sequence tensored with $V^*(1)$ one obtains the formula
		\begin{align*}
			\dim H^1_e(V^*(1)) = 
			\dim t_{V^*(1)}- \dim D_{cris}(V^*(1))^{\varphi=1} + h^0_{V^*(1)}.
		\end{align*}
		By orthogonality and using $$h^0_{V^*(1)}- h^1_{V} = -h^2_{V^*(1)} -[K:\QQ_p] \dim(V)$$
		we get 
		\begin{align*}
			-\dim  H^1_g(V) &= 
			\dim t_{V^*(1)}- \dim D_{cris}(V^*(1))^{\varphi=1} + h^0_{V^*(1)}- h^1_{V}  .\\
			&= \dim t_{V^*(1)}- \dim D_{cris}(V^*(1))^{\varphi=1} -h^2_{V^*(1)} -[K:\QQ_p] \dim(V)\\
			& = -\dim t_V - h^0(V)- \dim D_{cris}(V^*(1))^{\varphi=1},
		\end{align*}
		using in the last equation $h^2_{V^*(1)}=h^0_V,$ $\dim t_{V^*(1)}= [K:\QQ_p] \dim(V)-\dim D_{\mathrm{dR}}^+(V^*(1))$ and $\dim D_{\mathrm{dR}}^+(V^*(1)) = \dim t_V.$
		This leads to the well-known formula
		\begin{equation}\label{eq:classic} \dim H^1_g = \dim D_{\mathrm{dR}}(V)/D_{\mathrm{dR}}(V)^+ + \dim H^0(V) + \dim D_{cris}(V^*(1))^{\varphi=1}.\end{equation}
		Now consider for $\tilde{W}:= W^*(x^{-n-1})(1)$
		$$H^0(\tilde{W}) = \operatorname{ker}(\tilde{W}_e^{G_{K}}\oplus(\tilde{W}_{\mathrm{dR}}^+)^{G_{K}} \to \tilde{W}_{\mathrm{dR}}^{G_{K}} ).$$
		The Hodge--Tate weights are shifted precisely in a way, in which the map $(\tilde{W}_{\mathrm{dR}}^+)^{G_{K}} \to \tilde{W}_{\mathrm{dR}}^{G_{K}}$ is surjective, so the kernel of the above map is just $\tilde{W}_e^{G_{K}}  \cong \operatorname{D}_{\text{cris}}(V^*(1))^{\varphi=1}.$
		Using that twisting by $x^k$ does not change the $\fB_e$-component of a $B$-pair (by Lemma \ref{lem:deltatwist}) meaning that we have $\tilde{W}_e = W(V^*(1))_e = (\Bcris \otimes_{\QQ_p} V^{*}(1))^{\varphi=1}.$ 
	\end{proof}
	The classical formula \eqref{eq:classic} holds even for representations, which are not positive. In the case that $V$ is positive the first summand of \eqref{eq:classic} is zero and the relationship between the new formula from Corollary \ref{cor:newformula} and \eqref{eq:classic} is clarified by Proposition \ref{prop:classicformula}.
	
	Note that even though the objects in $\mathcal{BP}_E^{J,\mathbf{n}}$ behave nicely, they can still fail to be overconvergent. 
	\begin{prop} Assume $J \subsetneq \Sigma_K.$
		Let $\rho \colon K^\times \to E^\times$ be a character, which belongs to $\mathcal{BP}_E^{J,\mathbf{n}}$ but has at least two non-zero Hodge--Tate weights at the embeddings in $J$ \footnote{An explicit example of such a character would be $E(-1).$}. Then $H^1_{g,J}(B(\rho)) \neq 0$ and the extension corresponding to any non-zero class is not overconvergent. 
	\end{prop}
	\begin{proof} The fact that $H^1_{g,J}(B(\rho)) \neq 0$ follows from the Euler--Poincaré-Formula. By assumption $B(\rho)$ is not $K$-analytic. By \cite[Theorem 5.20]{FX12} any non-trivial extension of $E$ by $B(\rho)$ is not overconvergent. 
	\end{proof}
	It would be interesting to obtain $H^i_{g,J}(W)$ as the cohomology of a complex depending on $W.$ While this is the case for analytic cohomology, we do not know how to produce $H^i_{g,J}(W)$ as the cohomology of an explicit complex depending on $W$ without resorting to a construction with a categorical flavour such as in \cite{JTNB_2009__21_2_285_0}.

	\section{Analytic $(\varphi_K,\Gamma_K^{LT})$-modules} \label{sec:analytic}
	Fix an embedding $K \to E,$ which we denote henceforth by $\id.$
	Recall that a $B$-pair is called $K$-analytic if the representation $W_{\text{dR},\sigma}^+/tW_{\text{dR},\sigma}^+$ is trivial as a semi-linear $\CC_p$-representation for all $\sigma \neq \id.$ By Lemma \ref{lem:htW0derham} this is in fact equivalent to $W$ being $\sigma$-de Rham of weight $0,$ i.e., the category of $K$-analytic $B$-pairs is precisely $\mathcal{BP}_E^{\Sigma_K\setminus \id,(0,\dots,0)}.$
	In this section we collect some results concerning analytic $B$-pairs and their cohomology,
	which historically has been studied by analytic techniques. This leads to some subtle (and arguably pathological) problems when $K\neq \QQ_p.$ 
	Our goal is to present the bare minimum of theory which is needed to prove the comparison between said cohomology and the better behaved $H^i_{g,{\Sigma_K\setminus \id}}.$ 
    We will henceforth assume $K=E.$ This assumption is not necessary (for example the equivalence of categories in \cite{poyeton2023F} does not require this), but it is assumed in other references cited by us. The general case can be obtained by standard techniques via Galois descent.

	\subsection{Lubin--Tate $(\varphi_K,\Gamma_K^{LT})$-modules.}
	Let $K/\QQ_p$ be finite and $\varphi_K$ a Lubin--Tate power series attached to a uniformiser $\pi_K$ of $K.$ Let $K_\infty$ be the corresponding Lubin--Tate extension, obtained by adjoining the $\pi_K^n$-torsion points of the formal group attached to $\varphi_K$ and let $\Gamma_K^{\textrm{LT}}:= \Gal(K_\infty/K).$ To $V \in \operatorname{Rep}_K(G_K)$ one can also attach a Lubin--Tate $(\varphi_K,\Gamma_K^{\textrm{LT}})$-module over the $B_K:= \widehat{\mathcal{O}_K((T))}[1/p],$ where $\widehat{(-)}$-denotes $p$-adic completion. Providing again an equivalence between representations and étale $(\varphi_K,\Gamma_K^{\textrm{LT}})$-modules but contrary to the cyclotomic case the picture changes if we replace the coefficient ring $B_K$ by the Robba ring $\mathcal{R}_K$ consisting of Laurent series with coefficients in $K$ converging on some half-open annulus $r \leq \abs{x} <1$ for some $r \in (0,1).$ The category of étale $(\varphi_K,\Gamma_K^{\textrm{LT}})$-modules over $\cR_K$ is instead equivalent to the category of so-called overconvergent representations, which for $K\neq \QQ_p$ is a proper subcategory of $\operatorname{Rep}_K(G_K).$ 
	
	\subsection{$K$-analytic $B$-pairs}
	A sufficient condition for overconvergence is that the representation $V$ is $K$-analytic in the sense of Definition \ref{def:analytic}. In this case the $\Gamma_K^{\textrm{LT}}$ action on the attached $(\varphi_K,\Gamma_K^{\textrm{LT}})$-module $M$ over the Robba ring $\mathcal{R}_K$ is $K$-analytic in the sense that the derived action 
	$$\operatorname{Lie}\Gamma_K^{\textrm{LT}} \times M \to M$$ is $K$-bilinear. 
	 This can be promoted to an equivalence of categories: 
	 \begin{thm} \label{thm:equivalence}
	 	There is an equivalence of categories between the category of $K$-analytic $B$-pairs (in the sense of Definition \ref{def:analytic})
	 	and the category of  $(\varphi_K,\Gamma_K^{\textrm{LT}})$-modules over the Robba ring $\mathcal{R}_K,$ such that the derived $\Gamma_K^{\textrm{LT}}$-action is $K$-bilinear. 
	 \end{thm}
	
	\begin{proof}
The étale case is \cite[Theorem D]{Berger2016}. The $B$-pair case is \cite[Theorem A]{poyeton2023F}.
	\end{proof}
	\subsection{Analytic cohomology}
	For a $K$-analytic $(\varphi_K,\Gamma_K^{\mathrm{LT}})$-module $M$ over the Robba ring the action of $\Gamma_K^{\mathrm{LT}}$ extends to an action of the algebra $\mathcal{D}:=D(\Gamma_K^{\mathrm{LT}},K)$ of $K$-valued locally $K$-analytic distribution, i.e., the continuous dual of the space $C^{K-la}(\Gamma_K^{LT},K)$ of locally $K$-analytic functions $\Gamma_{K}^{LT} \to K$ (cf. \cite[Proposition 4.3.10.]{SchneiderVenjakobRegulator})

	 The analytic cohomology groups of $M$ can be defined as 
	$$H^i_{an,\varphi_L}(M):=\operatorname{Ext}^i_{\mathcal{D}[X]}(K,M),$$ where $X$ is a variable with $Xm:=(\varphi_K-1)m$ for $m \in M.$ Here the Ext groups are taken in the category of (abstract) $\mathcal{D}[X]$-modules. 
	There is another variant defined analogously using the canonical left-inverse operator $\Psi$ of $\varphi_K,$ which we denote by $H^i_{an,\Psi}(M).$
	\begin{rem}Typically in the literature analytic cohomology is defined either analogously to semi-group cohomology for $\varphi^\mathbb{N} \times \Gamma_K^{LT}$ but using so-called pro-analytic cocycles or as the cone of $\varphi-1$ acting on $\mathbf{R}\Hom_{\mathcal{D}}(K,M).$ They turn out to be isomorphic to our algebraic version. The comparison between the cocycle version and the second version is \cite[Corollary 4.2.7]{steingart2022finiteness}. The comparison between the second version and the one used here is \cite[Remark 4.3]{steingart2023comparisons}. 
	\end{rem}
	\begin{lem} \label{lem:appendixprops} Let $M$ be an $K$-analytic $(\varphi_K,\Gamma_K^{LT})$-module. Then
		\begin{enumerate} 
			\item $$H^i_{an,\varphi_K}(M) \cong H^i_{an,\Psi}(M)$$ for $i=0,1.$
			\item $$H^i_{an,\varphi_K}(M \widehat{\otimes}_{K}\CC_p) \cong H^i_{an,\Psi}(M \widehat{\otimes}_{K}\CC_p)$$ for $i=0,1,2.$
			\item 	For $i=0,1$ one has $H^i_{an}(M) \cong \operatorname{Ext}^i_{an}(\cR_K,M)$,
			where the right hand side denotes for $i=0$ the set of homomorphisms of $(\varphi_K,\Gamma_K^{LT})$-modules) from $\cR_K$ to $M$, and for $i=1$ the set of extensions of $\cR_K$ by $M,$ which are themselves $K$-analytic.
		\end{enumerate}
	\end{lem}
	\begin{proof}
		See \cite[Corollary 2.2.3]{BFFanalytic} for the first point. 
		 For the second point see \cite[Remark 3.2.2]{steingart2022finiteness}. The degree $0$ part of the last point is clear, the degree $1$ part is \cite[Theorem 4.2]{FX12}.
	\end{proof}
We are not sure, whether the $\varphi_K,$ and $\Psi$ version agree in degree $i=2$ over $K.$

	\subsection{An ad-hoc duality result}
	 
 It is technically convenient to use the $\Psi$-version because it is easy to show that $M/(\Psi-1)$ is finite-dimensional Hausdorff (cf. \cite[Proof of Lemma 5.7]{steingart2024iwasawa}), hence we will be using this version to prove an ad-hoc duality statement. 
In a certain sense, we are cheating as we are using the duality to reduce the comparison in degree $2$ to the degree $0$ (where nothing interesting happens) and we use the comparison with extension to compare in degree $i=1.$ 
	We believe that  $H^i_{\Sigma_K \setminus \id,g}$ is a better definition for $H^i_{an}$ as it avoids subtle questions such as the comparisons between the $\varphi_K$ and $\Psi$-version while being well-behaved and the duality deduced in the main Theorem is clearly related to Tate-duality, which is not obvious for the residue duality considered in \cite{SchneiderVenjakobRegulator}.

	\begin{prop} \label{prop:analyticdual}
		Let $M$ be an $K$-analytic $(\varphi_K,\Gamma_K^{LT})$-module and $\check{M} = M^*(\chi_C).$ The groups $H^i_{an,\Psi}(M)$ are finite $K$-dimensional and we have canonical  non-degenerate pairings 
		$$H^0_{an,\varphi}(M) \times H^2_{an,\Psi}(\check{M}) \to K$$
			$$H^0_{an,\varphi}(\check{M}) \times H^2_{an,\Psi}(M) \to K$$
		In particular 
		$$H^0_{an,\Psi}(M) \cong H^0_{an,\varphi}(M)  \cong \operatorname{Hom}_K( H^2_{an,\psi}(\check{M}),K)$$
		and
		$$H^2_{an,\Psi}(M) \cong \operatorname{Hom}_K(H^0_{an,\varphi}(\check{M}),K)= \operatorname{Hom}_K(H^0_{an,\Psi}(\check{M}),K)  $$
	\end{prop}
	\begin{proof}
		Since $\check{\check{M}}=M$ it suffices to consider one of the two pairings. 
		
		We have  
		 $\operatorname{Ext}^2_{\mathcal{D}[Y]}(K,M) = \operatorname{Ext}^1_{\mathcal{D}}(K,M/(\Psi-1)).$
		 This follows by homological algebra using \cite[Remark 4.3 (i)]{steingart2023comparisons} and the fact that 
		 $\operatorname{Ext}^2_{\mathcal{D}}(K,M^{\Psi=1}) = 0$ by using the comparison with Lie-algebra cohomology \cite[Theorem 4.10]{Kohlhaase} (noting that by Theorem 6.5 of \textit{(loc.cit)} the result is applicable to the abelian hence solvable group $\Gamma_K^{LT}$). 
		 Using the same comparison result we also get 
		 $$\operatorname{Ext}^1_{\mathcal{D}}(K,M/(\Psi-1)) \cong H^0(\Gamma_K^{LT},M/((\Psi-1),\nabla)).$$
		 Because $M/(\Psi-1)$ is finite-dimensional and $\nabla$ acts trivially, we conclude that the action of $\Gamma_L$ is discrete, which in characteristic $0$ means that  
		 $$ H^0(\Gamma_K^{LT},M/((\Psi-1),\nabla)) \cong M/(\Psi-1,\Gamma_K^{LT}).$$
		Consider the canonical residue pairing 
		$ \langle \cdot,\cdot \rangle \colon M \times \check{M} \to K$ from \cite[Lemma 4.5.1]{SchneiderVenjakobRegulator}.
	Since $\check{M}/(\Psi-1,\Gamma_K^{LT})$ is Hausdorff and finite dimensional any linear functional is continuous. Let $\mu \in \operatorname{Hom}_K( H^2_{an,\psi}(\check{M}),K)$. This defines a continuous linear functional on $\check{M}$ satisfying $\mu(y)=0$ for any $y$ of the form $y= (\Psi-1)x$ or $y= (\gamma-1)x$ with $x \in \check{M}$ and $\gamma \in \Gamma_K^{LT}.$ From the pairing we get that $\mu$ is given by evaluation $\langle m,- \rangle$ for some $m \in M.$ We have (using the properties shown in \textit{(loc.cit)}) 
	$$\langle (\varphi-1) m,x\rangle = \langle m ,(\Psi-1)x\rangle =0$$ respectively 
	$$\langle (\gamma-1) m,x\rangle = \langle m ,(\gamma^{-1}-1)x\rangle =0$$ for any $x \in M.$
	Which by non-degenerateness of the pairing implies $m \in H^0_{an,\varphi_K}(M) \cong H^0_{an,\Psi}(M).$
	It is easy to check that the assignment $\mu \mapsto m$ defines an $K$-linear map 
	$$ \operatorname{Hom}_K( H^2_{an,\psi}(\check{M}),K) \to H^0_{an,\varphi_K}(M).$$
	
	By the same arguments in reverse order we see that sending $m \in  H^0_{an,\varphi_K}(M) $ to the evaluation at $m$ defines a map $H^0_{an,\varphi_K}(M) \to \operatorname{Hom}_K( H^2_{an,\psi}(\check{M}),K),$ which is an inverse.
	\end{proof}
	We can now finish by proving the last part of Theorem \ref{thm:mainresultintro} from the introduction.
	\begin{prop} Let $B$ be a $K$-analytic $B$-pair, i.e., $B$ belongs to $\mathcal{BP}^{\mathbf{n},J}$ with 
			 $J = \Sigma_K \setminus \sigma_0$ and $\mathbf{n}= (0,\dots,0).$Then for the corresponding $K$-analytic $(\varphi_K,\Gamma_K^{LT})$-module $M:=M(W)$ we have 
			 
			  $$H^i_{g,J}(W) \cong H^i_{an,\Psi}(M)$$ for every $i.$
	\end{prop}
	\begin{proof}
			By combining the duality results from Proposition \ref{prop:analyticdual} and the duality results in Theorem \ref{thm:mainresult} we can reduce to showing 
		$$H^0_{g,J}(W) \cong H^0_{\text{an}}(M)$$
		and $$H^1_{g,J}(W) \cong H^1_{\text{an}}(M).$$
		These follow by comparing $H^i(X)$ on both sides with  homomorphisms of the trivial object into $X$ (resp. extensions of the trivial object by $X$) in their respective categories using Lemma \ref{lem:appendixprops} and Theorem \ref{thm:mainresult} respectively. 
	\end{proof}

	\let\stdthebibliography\thebibliography
	\let\stdendthebibliography\endthebibliography
	\renewenvironment*{thebibliography}[1]{%
		\stdthebibliography{MSVW25}}
	{\stdendthebibliography}

	\bibliographystyle{amsalpha}
	
	\bibliography{Literatur}
\end{document}